\documentclass{amsart}
\usepackage{amsmath, amssymb, amsthm, epsfig}
\usepackage{amssymb, latexsym}
\usepackage{url}

\DeclareMathOperator{\CM}{\mathrm{CM}}

\DeclareMathOperator{\norm}{\mathrm{norm}}
\DeclareMathOperator{\Norm}{\mathrm{Norm}}

\DeclareMathOperator{\rank}{\mathrm{rank}}

\DeclareMathOperator{\Tor}{\mathrm{Tor}}

\begin{document}
 \bibliographystyle{plain}

 \newtheorem{theorem}{Theorem}[section]
 \newtheorem{lemma}{Lemma}[section]
 \newtheorem{corollary}{Corollary}[section]
 \newtheorem{conjecture}{Conjecture}
 \newtheorem{definition}{Definition}
 
 \newcommand{\mc}{\mathcal}
 \newcommand{\A}{\mc A}
 \newcommand{\B}{\mc B}
 \newcommand{\cc}{\mc C}
 \newcommand{\D}{\mc D}
 \newcommand{\E}{\mc E}
 \newcommand{\F}{\mc F}
 \newcommand{\G}{\mc G}
 \newcommand{\hH}{\mc H}
 \newcommand{\I}{\mc I}
 \newcommand{\J}{\mc J}
 \newcommand{\K}{\mc K}
 \newcommand{\eL}{\mc L}
 \newcommand{\M}{\mc M}
 \newcommand{\eN}{\mc N}
 \newcommand{\pp}{\mc P}
 \newcommand{\qq}{\mc Q}
 \newcommand{\U}{\mc U}
 \newcommand{\V}{\mc V}
 \newcommand{\W}{\mc W}
 \newcommand{\X}{\mc X}
 \newcommand{\Y}{\mc Y}
 \newcommand{\zZ}{\mc Z}
 \newcommand{\C}{\mathbb{C}}
 \newcommand{\R}{\mathbb{R}}
 \newcommand{\Q}{\mathbb{Q}}
 \newcommand{\T}{\mathbb{T}}
 \newcommand{\Z}{\mathbb{Z}}
 \newcommand{\aA}{\mathfrak A}
 \newcommand{\bB}{\mathfrak B}
 \newcommand{\cC}{\mathfrak C}
 \newcommand{\dD}{\mathfrak D}
 \newcommand{\ee}{\mathfrak E}
 \newcommand{\ff}{\mathfrak F}
 \newcommand{\iI}{\mathfrak I}
 \newcommand{\mM}{\mathfrak M}
 \newcommand{\nN}{\mathfrak N}
 \newcommand{\pP}{\mathfrak P}
 \newcommand{\uU}{\mathfrak U}
 \newcommand{\fb}{f_{\beta}}
 \newcommand{\fg}{f_{\gamma}}
 \newcommand{\gb}{g_{\beta}}
 \newcommand{\vphi}{\varphi}
 \newcommand{\p}{\varphi}
 \newcommand{\ep}{\varepsilon}
 \newcommand{\bo}{\boldsymbol 0}
 \newcommand{\bone}{\boldsymbol 1}
 \newcommand{\balpha}{\boldsymbol \alpha}
 \newcommand{\bb}{\boldsymbol b}
 \newcommand{\bc}{\boldsymbol c}
 \newcommand{\be}{\boldsymbol e}
 \newcommand{\bff}{\boldsymbol f}
 \newcommand{\bk}{\boldsymbol k}
 \newcommand{\bl}{\boldsymbol l}
 \newcommand{\bm}{\boldsymbol m}
 \newcommand{\bn}{\boldsymbol n}
 \newcommand{\bgamma}{\boldsymbol \gamma}
 \newcommand{\blambda}{\boldsymbol \lambda}
 \newcommand{\bq}{\boldsymbol q}
 \newcommand{\bt}{\boldsymbol t}
 \newcommand{\bu}{\boldsymbol u}
 \newcommand{\bv}{\boldsymbol v}
 \newcommand{\bw}{\boldsymbol w}
 \newcommand{\bx}{\boldsymbol x}
 \newcommand{\bwy}{\boldsymbol y}
 \newcommand{\bnu}{\boldsymbol \nu}
 \newcommand{\bxi}{\boldsymbol \xi}
 \newcommand{\bz}{\boldsymbol z}
 \newcommand{\whG}{\widehat{G}}
 \newcommand{\oK}{\overline{K}}
 \newcommand{\oKt}{\overline{K}^{\times}}
 \newcommand{\oQ}{\overline{\Q}}
 \newcommand{\oq}{\oQ^{\times}}
 \newcommand{\oQt}{\oQ^{\times}/\Tor\bigl(\oQ^{\times}\bigr)}
 \newcommand{\ot}{\Tor\bigl(\oQ^{\times}\bigr)}
 \newcommand{\h}{\frac12}
 \newcommand{\hh}{\tfrac12}
 \newcommand{\dx}{\text{\rm d}x}
 \newcommand{\dbx}{\text{\rm d}\bx}
 \newcommand{\dy}{\text{\rm d}y}
 \newcommand{\dmu}{\text{\rm d}\mu}
 \newcommand{\dnu}{\text{\rm d}\nu}
 \newcommand{\dla}{\text{\rm d}\lambda}
 \newcommand{\dlav}{\text{\rm d}\lambda_v}
 \newcommand{\trho}{\widetilde{\rho}}
 \newcommand{\dtrho}{\text{\rm d}\widetilde{\rho}}
 \newcommand{\drho}{\text{\rm d}\rho}

\title[heights]{On the height of solutions\\to norm form equations}
\author{Shabnam Akhtari and Jeffrey~D.~Vaaler}
\thanks{}
\medskip

\address{Department of Mathematics, University of Oregon, Eugene, Oregon 97402 USA}
\email{akhtari@uoregon.edu}
\medskip

\address{Department of Mathematics, University of Texas, Austin, Texas 78712 USA}
\email{vaaler@math.utexas.edu}

\thanks{Shabnam Akhtari's research is funded by the NSF grant DMS-1601837.}

\begin{abstract}  Let $k$ be a number field.  We consider norm form equations associated to a full $O_k$-module
contained in a finite extension field $l$.  It is known that the set of solutions is naturally a union of disjoint equivalence classes of 
solutions.  We prove that each nonempty equivalence class of solutions contains a representative with Weil height 
bounded by an expression that depends on parameters defining the norm form equation.
\end{abstract}
\maketitle
\numberwithin{equation}{section}

\section{Introduction}
 
 Classically norm form equations are defined over the field of rational numbers. 
Let $\omega_1, \omega_2, \dots , \omega_N$, be points in $\oQ$ that are $\Q$-linearly independent, and let 
\begin{equation*}\label{norm0}
K = \Q(\omega_1, \omega_2, \dots , \omega_N)
\end{equation*} 
be the algebraic number field that they generate.  We
assume that $[K : \Q] = d$, and we write $\sigma_1, \sigma_2, \dots , \sigma_d$, for the distinct embeddings of 
$K$ into $\oQ$.  Using $\omega_1, \omega_2, \dots , \omega_N$, we define a homogeneous polynomial in a vector
variable $\bx$ having $N$ independent coordinates $x_1, x_2, \dots , x_N$, by
\begin{align}\label{norm1}
G(\bx) = \prod_{i=1}^d \bigg\{\sum_{n=1}^N \sigma_i(\omega_n) x_n\bigg\}.
\end{align}
It is easy to verify that $G(\bx)$ has rational coefficients and, as $\omega_1, \omega_2, \dots , \omega_N$, are $\Q$-linearly 
independent, $G(\bx)$ is not identically zero.  The homogeneous polynomial $G(\bx)$ is called a {\it norm form}, 
because if $\bxi$ is a nonzero point with rational integer coordinates $\xi_1, \xi_2, \dots , \xi_N$, then
\begin{equation*}
G(\bxi) = \Norm_{K/\Q}\bigl(\omega_1 \xi_1 + \omega_2 \xi_2 + \cdots + \omega_N \xi_N\bigr),
\end{equation*}
where
\begin{equation*}
\Norm_{K/\Q} : K^{\times} \rightarrow \Q^{\times}
\end{equation*}
is the norm homomorphism. 
In  \cite{schmidt1971} Schmidt proved his fundamental result, that a norm form equation $G(\bx) = b$, where 
$b \in \mathbb{Q}$,  has only finitely
many solutions if $G$ satisfies some natural non-degeneracy condition. Later, in another breakthrough work 
\cite{schmidt1972}, Schmidt dealt also with the case that $G$ is degenerate and showed that in that case, the set of solutions 
of the norm form equation can be partitioned in a natural way into families, and is the union of finitely many such 
families.   This was soon followed by an analogous $p$-adic result due to Schlickewei \cite{schlickewei1977}.
Schmidt's results have been generalized in different interesting ways (e.g. \cite{Evertse00, Evertse95, EvGy97}), in particular 
Laurent, in \cite{laurent1984}, considered norm form equations into $k$, a finite algebraic extension of $\mathbb{Q}$.

Let $k$ and $l$ be algebraic number fields such that
\begin{equation*}\label{intro1}
\Q \subseteq k \subseteq l \subseteq \oQ,
\end{equation*}
where $\oQ$ is an algebraic closure of $\Q$.  We write $k^{\times}$ and $l^{\times}$ for the multiplicative group of nonzero
elements in $k$ and $l$, respectively, and
\begin{equation}\label{intro5}
\Norm_{l/k} : l^{\times} \rightarrow k^{\times}
\end{equation}
for the norm homomorphism.  We also write $O_k$ for the ring of algebraic integers in $k$, $O_k^{\times}$ for the
multiplicative group of units in $O_k$, and $\Tor\bigl(O_k^{\times}\bigr)$ for the finite group of roots of unity in $O_k^{\times}$.
Then $O_l$, $O_l^{\times}$, and $\Tor\bigl(O_l^{\times}\bigr)$ are the analogous subsets in $l$.  

Let $\omega_1, \omega_2, \dots , \omega_e$, be $k$-linearly independent elements of $l$ that form a basis for $l$ as a
$k$-vector space, and let $\sigma_1, \sigma_2, \dots , \sigma_e$, be the collection of distinct embeddings of 
$l$ into $\oQ$ that fix the subfield $k$.  It follows that 
\begin{equation}\label{intro20}
F(\bx) = \prod_{i =1}^e \bigl(\sigma_i(\omega_1) x_1 + \sigma_i(\omega_2) x_2 + \cdots + \sigma_i(\omega_e) x_e\bigr)
\end{equation}
is a homogeneous polynomial of degree $e$ in independent variables $x_1, x_2, \dots , x_e$, and the coefficients of $F$ belong 
to the field $k$.  The homogeneous polynomial $F(\bx)$ defined by (\ref{intro20}) is an example of a {\it norm form}.
For $\beta \not= 0$ in $k$, we consider the norm form equation
\begin{equation}\label{intro25}
F(\bx) = \zeta \beta,\quad\text{where $\zeta \in \Tor\bigl(O_k^{\times}\bigr)$},
\end{equation}
and we seek to describe the solutions in $(O_k)^e$.

Rather than working with the polynomial $F$ defined by (\ref{intro20}), we will work instead with the full $O_k$-module
\begin{equation}\label{intro35}
\mM = \big\{\omega_1 \nu_1 + \omega_2 \nu_2 + \cdots + \omega_e \nu_ e : 
				\text{$\nu_i \in O_k$ for $i = 1, 2, \dots , e$}\big\}
\end{equation}
generated by the basis $\omega_1, \omega_2, \dots , \omega_e$.  If $\bnu = (\nu_i)$ is a nonzero point in $(O_k)^e$, we have
\begin{equation*}\label{intro40}
F(\bnu) = \Norm_{l/k}(\mu),
\end{equation*}
where
\begin{equation*}\label{intro45}
\mu = \omega_1 \nu_1 + \omega_2 \nu_2 + \cdots + \omega_e \nu_ e
\end{equation*}
belongs to the full $O_k$-module $\mM$.  Thus for $\beta \not= 0$ in $k$, we wish to describe the set of solutions
\begin{equation}\label{intro55}
\big\{\mu \in \mM : \Norm_{l/k}(\mu) \in \Tor\bigl(O_k^{\times}\bigr) \beta\big\}.
\end{equation}
There is a natural equivalence relation in $\mM \setminus \{0\}$, such that the set (\ref{intro55}) is either 
empty, or it is a disjoint union of finitely many equivalence classes.

If $\alpha \not= 0$ belongs to $l$, then $\alpha \omega_1, \alpha \omega_2, \dots , \alpha \omega_e$, is also
a basis for $l$ as a $k$-vector space.  This second basis generates the full $O_k$-module
\begin{equation*}\label{intro56}
\alpha \mM = \big\{\alpha \omega_1 \nu_1 + \alpha \omega_2 \nu_2 + \cdots + \alpha \omega_e \nu_ e : 
				\text{$\nu_i$ in $O_k$ for $i = 1, 2, \dots , e$}\big\}.
\end{equation*}
We say that the $O_k$-modules $\mM$ and $\alpha \mM$ are {\it proportional}.  It is obvious that proportionality is an
equivalence relation in the collection of all full $O_k$-modules contained in $l$.  As
\begin{equation*}\label{intro57}
\Norm_{l/k}(\alpha \mu) = \Norm_{l/k}(\alpha) \Norm_{l/k}(\mu),
\end{equation*}
the problem of describing the solution set (\ref{intro55}) changes insignificantly if the $O_k$-module $\mM$ is replaced 
by a proportional $O_k$-module $\alpha \mM$.  Each proportionality class plainly contains a representative that is a 
subset of $O_l$.  Therefore in the remainder of this paper we assume that $\mM \subseteq O_l$.  With this 
assumption we can restrict our attention to solution sets (\ref{intro55}) such that $\beta \not= 0$ also belongs to $O_k$.

The {\it coefficient ring} associated to the full module $\mM$ is the subset
\begin{equation}\label{intro60}
O_{\mM} = \big\{\alpha \in l : \alpha \mM \subseteq \mM\big\}.
\end{equation}
It is easy to check that proportional $O_k$-modules contained in $l$ have the same coefficient ring. 
Let $\psi_1, \psi_2, \dots , \psi_f$, be an integral basis for $O_k$.  It follows that
\begin{equation*}\label{intro61}
e = [l : k],\quad f = [k : \Q],
\end{equation*}
and that
\begin{equation}\label{intro63}
\big\{\omega_i \psi_j : \text{$i = 1, 2, \dots , e$, and $j = 1, 2, \dots , f$}\big\}
\end{equation}
is a basis for $\mM$ as a full $\Z$-module in $l$.  Therefore we can appeal to classical results such as
\cite[Chap. 2, Sec. 2, Theorem 3]{BS1966}, and conclude that the coefficient ring $O_{\mM}$ is an order in $l$.
We recall (see \cite[Chapter 5, section 1]{Weil}) that $O_l$ is the maximal order in $l$, so that
\begin{equation*}\label{intro65}
O_{\mM} \subseteq O_l.
\end{equation*}

Let $r(l)$ be the rank of the group $O_l^{\times}$ of units in $O_l$, and let $r(k)$ be the rank of $O_k^{\times}$.
By the extension of Dirichlet's unit theorem to orders, the subgroup 
\begin{equation*}\label{intro68}
O_{\mM}^{\times} = O_{\mM} \cap O_l^{\times}
\end{equation*}
of units in $O_{\mM}$ has rank $r(l)$, and therefore the index $\bigl[O_l^{\times} : O_{\mM}^{\times}\bigr]$ is finite.  
And it follows from (\ref{intro60}) that
\begin{equation}\label{intro70}
O_{\mM}^{\times} = \big\{\alpha \in l : \alpha \mM = \mM\big\}.
\end{equation}
Hence the group $O_{\mM}^{\times}$ acts on the module $\mM$ by multiplication.  Let
\begin{equation}\label{intro80}
\E_{l/k}(\mM) = \big\{\alpha \in O_{\mM}^{\times} : \Norm_{l/k}(\alpha) \in \Tor\bigl(O_k^{\times}\bigr)\big\}
\end{equation}
be the subgroup of {\it relative units} in the coefficient ring $O_{\mM}$.   In Lemma \ref{lemprelim1} we show that the 
subgroup $\E_{l/k}(\mM)$ has rank
\begin{equation*}\label{intro82}
r(l/k) = r(l) - r(k).
\end{equation*}

Now suppose that $\beta \not= 0$ belongs to $O_k$, and $\mu$ in $\mM$ satisfies
\begin{equation}\label{intro85}
\Norm_{l/k}(\mu) = \zeta \beta, \quad\text{where $\zeta \in \Tor\bigl(O_k^{\times}\bigr)$}.
\end{equation}
If $\gamma$ belongs to the group $\E_{l/k}(\mM)$ of relative units in $O_{\mM}$, then (\ref{intro70}) implies that $\gamma \mu$ 
belongs to $\mM$.  And it follows from (\ref{intro80}) that
\begin{equation}\label{intro90}
\Norm_{l/k}(\gamma \mu) = \Norm_{l/k}(\gamma) \zeta \beta = \zeta^{\prime} \beta,
						\quad\text{where $\zeta^{\prime} \in \Tor\bigl(O_k^{\times}\bigr)$}.
\end{equation}
We say that two nonzero elements $\mu_1$ and $\mu_2$ in $\mM$ are {\it equivalent} if there exists an element $\gamma$ in
the group $\E_{l/k}(\mM)$ such that $\gamma \mu_1 = \mu_2$.  It is trivial that this is an equivalence relation in 
$\mM \setminus \{0\}$.  Indeed, each equivalence class is also a coset in the quotient group $l^{\times}/\E_{l/k}(\mM)$.
It follows from (\ref{intro85}) and (\ref{intro90}) that for each $\beta \not= 0$ in $O_k$, the set 
(\ref{intro55}) is a disjoint union of equivalence classes.  It is known that (\ref{intro55}) is a disjoint union of finitely many 
such equivalence classes (see \cite{laurent1984, schmidt1972}).  A finiteness result of this sort also follows from 
Northcott's theorem \cite{northcott1949} 
(see also \cite[Theorem 1.6.8]{bombieri2006}) and the following inequality.  Here we write $\alpha \mapsto h(\alpha)$ for 
the Weil height of an algebraic number $\alpha \not= 0$, and we define this explicitly in (\ref{unit35}).

\begin{theorem}\label{thmintro1}  Let the full $O_k$-module $\mM \subseteq O_l$ be defined by {\rm (\ref{intro35})}, 
and assume that the rank $r(l/k)$ of the group $\E_{l/k}(\mM)$ of relative units is positive.  Let
\begin{equation}\label{intro95}
\ep_1, \ep_2, \dots , \ep_{r(l/k)},
\end{equation}
be multiplicatively independent units in the subgroup $\E_{l/k}(\mM)$.
Assume that $\beta \not= 0$ is a point in $O_k$, and $\mu \not= 0$ is a point in $\mM$, such that
\begin{equation}\label{intro100}
\Norm_{l/k}(\mu) = \zeta \beta,\quad\text{where $\zeta \in \Tor\bigl(O_k^{\times}$}\bigr).
\end{equation}
Then there exists an element $\gamma$ in $\E_{l/k}(\mM)$, such that $\gamma \mu$ belongs to $\mM$,
\begin{equation}\label{intro105}
\Norm_{l/k}(\gamma \mu) = \zeta^{\prime} \beta,\quad\text{where $\zeta^{\prime} \in \Tor\bigl(O_k^{\times}$}\bigr),
\end{equation}
and
\begin{equation}\label{intro110}
h(\gamma \mu) \le \hh \sum_{j = 1}^{r(l/k)} h(\ep_j) + [l : k]^{-1} h(\beta).
\end{equation}
\end{theorem}

So as to give a complete treatment of the problem considered here, we also prove the following much simpler result.

\begin{theorem}\label{thmintro2}  Let the full $O_k$-module $\mM \subseteq O_l$ be defined by {\rm (\ref{intro35})}, 
and assume that the rank $r(l/k)$ of the group $\E_{l/k}(\mM)$ of relative units is zero.  
Assume that $\beta \not= 0$ is a point in $O_k$, and $\mu \not= 0$ is a point in $\mM$, such that
\begin{equation}\label{intro130}
\Norm_{l/k}(\mu) = \zeta \beta,\quad\text{where $\zeta \in \Tor\bigl(O_k^{\times}$}\bigr).
\end{equation}
Then we have
\begin{equation}\label{intro135}
h(\mu) = [l : k]^{-1} h(\beta).
\end{equation}
\end{theorem}

\section{The rank of the group of relative units}\label{rank}

Following Costa and Friedman \cite{costa1991} and \cite{costa1993}, the subgroup of relative units in $O_l^{\times}$ 
with respect to the subfield $k$, is defined by
\begin{equation*}\label{prelim1}
\E_{l/k} = \big\{\alpha \in O_l^{\times} : \Norm_{l/k}(\alpha) \in \Tor\bigl(O_k^{\times}\bigr)\big\}.
\end{equation*}
Hence the subgroup of relative units in $O_{\mM}$ is
\begin{equation*}\label{prelim5}
\E_{l/k}(\mM) = \E_{l/k} \cap O_{\mM}^{\times}.
\end{equation*}
Costa and Friedman show that $\E_{l/k}$ has rank $r(l) - r(k)$ (see also \cite[section 3]{akhtari2015}).  Here
we show that the subgroup $\E_{l/k}(\mM)$ also has rank $r(l) - r(k)$.

\begin{lemma}\label{lemprelim1}  Let the full $O_k$-module $\mM \subseteq O_l$ be defined by {\rm (\ref{intro35})}, 
and let the subgroup $\E_{l/k}(\mM)$ of relative units in $O_{\mM}^{\times}$ be defined by {\rm (\ref{intro80})}.  Then
the rank of $\E_{l/k}(\mM)$ is
\begin{equation}\label{prelim10}
r(l/k) = r(l) - r(k).
\end{equation}
\end{lemma}

\begin{proof}  As $O_{\mM}$ is an order in $l$, it follows from the extension of Dirichlet's unit theorem to orders 
(see \cite[Chap. 2, Sec. 4, Theorem 5]{BS1966}) that the group of units $O_{\mM}^{\times}$ has rank $r(l)$.  The 
norm (\ref{intro5}) restricted to $O_{\mM}^{\times}$ is a homomorphism
\begin{equation*}\label{prelim20}
\Norm_{l/k} : O_{\mM}^{\times} \rightarrow O_k^{\times},
\end{equation*}
and the norm restricted to the torsion subgroup $\Tor\bigl(O_{\mM}^{\times}\bigr)$ is a homomorphism
\begin{equation*}\label{prelim25}
\Norm_{l/k} : \Tor\bigl(O_{\mM}^{\times}\bigr) \rightarrow \Tor\bigl(O_k^{\times}\bigr).
\end{equation*}
Hence we get a well defined homomorphism, which we write as
\begin{equation}\label{prelim30}
\norm_{l/k} : O_{\mM}^{\times}/ \Tor\bigl(O_{\mM}^{\times}\bigr) \rightarrow O_k^{\times}/\Tor\bigl(O_k^{\times}\bigr),
\end{equation}
by setting
\begin{equation*}\label{prelim35}
\norm_{l/k}\bigl(\alpha \Tor\bigl(O_{\mM}^{\times}\bigr)\bigr) = \Norm_{l/k}(\alpha) \Tor\bigl(O_k^{\times}\bigr).
\end{equation*}
To simplify notation we write
\begin{equation*}\label{prelim40}
\F_{\mM} = O_{\mM}^{\times}/ \Tor\bigl(O_{\mM}^{\times}\bigr),\quad
							\text{and}\quad \F_k = O_k^{\times}/\Tor\bigl(O_k^{\times}\bigr),
\end{equation*}
and we use coset representatives rather than cosets for points in $\F_{\mM}$ and $\F_k$.  It is clear that 
$\F_{\mM}$ and $\F_k$ are free groups such that
\begin{equation*}\label{prelim45}
\rank \F_{\mM} = \rank O_{\mM}^{\times} = r(l),\quad\quad \rank \F_k = \rank O_k^{\times} = r(k),
\end{equation*}
and
\begin{equation*}\label{prelim50}
\norm_{l/k} : \F_{\mM} \rightarrow \F_k.
\end{equation*}
We note that the image of the subgroup
\begin{equation*}\label{prelim55}
\E_{l/k}(\mM) = \big\{\alpha \in O_{\mM}^{\times} : \Norm_{l/k}(\alpha) \in \Tor\bigl(O_k^{\times}\bigr)\big\}
\end{equation*}
of relative units in the group $\F_{\mM}$, is the kernel
\begin{equation}\label{prelim60}
\big\{\alpha \in \F_{\mM} : \norm_{l/k} (\alpha) = 1\big\}
\end{equation} 
of the homomorphism $\norm_{l/k}$.  Thus it suffice to show that the kernel (\ref{prelim60}) has rank $r(l) - r(k)$.

Let $\vphi_1, \vphi_2, \dots , \vphi_{r(k)}$ be multiplicatively independent elements in the group $O_k^{\times}$.
As $O_k^{\times} \subseteq O_l^{\times}$ and the index $[O_l^{\times} : O_{\mM}^{\times}]$ is finite, there exist
positive integers $m_1, m_2, \dots , m_{r(k)}$ such that
\begin{equation*}\label{prelim65}
\vphi_1^{m_1}, \vphi_2^{m_2}, \dots , \vphi_{r(k)}^{m_{r(k)}}
\end{equation*}
are multiplicatively independent elements in 
\begin{equation*}\label{prelim70}
O_{\mM}^{\times} \cap O_k^{\times}.
\end{equation*}
For each $j = 1, 2, \dots , r(k)$ we have
\begin{equation*}\label{prelim75}
\norm_{l/k}\bigl(\vphi_j^{m_j}\bigr) = \vphi_j^{e m_j},
\end{equation*}
where $e = [l : k]$.  Hence the image of the  homomorphism $\norm_{l/k}$ in $O_k^{\times}$ has rank $r(k)$.
It follows that the kernel of the homomorphism $\norm_{l/k}$ has rank $r(l) - r(k)$.  We have already noted that
this is the rank of $\E_{l/k}(\mM)$, and so the proof of the lemma is complete.
\end{proof}

\section{Inequalities for relative units}\label{IRU}

At each place $w$ of $l$ we write $l_w$ for the completion of $l$ at $w$, so that $l_w$ is a local field.  We select two 
absolute values $\|\ \|_w$ and $|\ |_w$ from the place $w$.  The absolute value $\|\ \|_w$ extends the usual archimedean or 
nonarchimedean absolute value on the subfield $\Q$.  Then $|\ |_w$ must be a power of $\|\ \|_w$, and we set
\begin{equation}\label{unit30}
|\ |_w = \|\ \|_w^{d_w/d},
\end{equation}
where $d_w = [l_w : \Q_w]$ is the local degree of the extension, and $d = [l : \Q]$ is the global degree.  With these normalizations
the {\it Weil height} (or simply the height) is a function
\begin{equation*}\label{unit32}
h : l^{\times} \rightarrow [0, \infty)
\end{equation*}
defined at each algebraic number $\alpha$ that belongs to $l^{\times}$, by
\begin{equation}\label{unit35}
h(\alpha) = \sum_w \log^+ |\alpha|_w = \hh \sum_w \bigl|\log |\alpha|_w\bigr|.
\end{equation}
Each sum in (\ref{unit35}) is over the set of all places $w$ of $l$, and the equality between the two sums follows 
from the product formula.  Then $h(\alpha)$ depends on the algebraic number $\alpha \not= 0$, but it does not depend on 
the number field $l$ that contains $\alpha$.  It is often useful to recall that the height is constant on each coset of the
quotient group $l^{\times}/\Tor\bigl(l^{\times}\bigr)$, and therefore we have $h(\zeta \alpha) = h(\alpha)$ for each
element $\alpha$ in $l^{\times}$, and each root of unity $\zeta$ in $\Tor\bigl(l^{\times}\bigr)$.  Elementary properties of the 
height (see \cite{bombieri2006} for further details) imply that the map $(\alpha, \beta) \mapsto h\bigl(\alpha \beta^{-1}\bigr)$ 
defines a metric on the group $l^{\times}/\Tor\bigl(l^{\times}\bigr)$.  If $\alpha$ belongs to the subgroup $k^{\times}$, we have
\begin{equation}\label{unit34}
h(\alpha) = \hh \sum_w \bigl|\log |\alpha|_w\bigr| = \hh \sum_v \bigl|\log |\alpha|_v\bigr|,
\end{equation}
where the sum on the right of (\ref{unit34}) is over the set of all places $v$ of $k$, and the absolute values $|\ |_v$ are 
normalized with respect to $k$.  We write $v$ for a place of $k$, and use $w$ or $x$ for a place of $l$.  
Additional properties of the Weil height on groups are discussed in \cite{akhtari2015}, and \cite{vaaler2014}.

For each place $v$ of $k$ we write
\begin{equation*}\label{unit38}
W_v(l/k) = \big\{w : \text{$w$ is a place of $l$ and $w | v$}\big\}.
\end{equation*}
The set $W_{\infty}(l/\Q)$ of archimedean (or infinite) places of $l$ has cardinality $r(l) + 1$, and similarly the set $W_{\infty}(k/\Q)$ 
has cardinality $r(k) + 1$.  Let $\R^{r(l) + 1}$ denote the real vector space of (column) vectors $\bxi = (\xi_w)$ with 
coordinates indexed by places $w$ in $W_{\infty}(l/\Q)$.  We define
\begin{equation}\label{unit43}
\D_{r(l/k)} = \Big\{\bxi \in \R^{r(l) + 1} : \text{$\sum_{w | v} \xi_w = 0$ for each $v$ in $W_{\infty}(k/\Q)$}\Big\},
\end{equation}
so that $\D_{r(l/k)}$ is a subspace of dimension 
\begin{equation*}\label{unit45}
\bigl(r(l) + 1\bigr) - \bigl(r(k) +1\bigr) = r(l) - r(k) = r(l/k).
\end{equation*}
contained in $\R^{r(l) + 1}$.  Let $\eta_1, \eta_2, \dots , \eta_{r(l/k)}$, be a fundamental system of units for the subgroup 
$\E_{l/k}(\mM)$ of relative units in $O_{\mM}^{\times}$, so that
\begin{equation*}\label{unit46}
\E_{l/k}(\mM) = \Tor\bigl(\E_{l/k}(\mM)\bigr) \otimes \big\langle \eta_1, \eta_2, \dots , \eta_{r(l/k)} \big\rangle.
\end{equation*}
Then let $L$ denote the $(r(l) + 1) \times r(l/k)$ real matrix
\begin{equation}\label{unit48}
L = \bigl(\log |\eta_j|_w\bigr),
\end{equation}
where $w$ in $W_{\infty}(l/\Q)$ indexes rows, and $j = 1, 2, \dots , r(l/k)$, indexes columns.  Because the relative regulator does 
not vanish (see \cite{costa1991} and \cite{costa1993}), it follows that the matrix $L$ has $\R$-rank equal to $r(l/k)$.  Then 
using the product formula we find that
\begin{equation}\label{unit53}
\bwy \mapsto L\bwy = \biggl(\sum_{j = 1}^s y_j \log |\eta_j|_v \biggr)
\end{equation}
is a linear map from the $\R$-linear space
\begin{equation}\label{unit58}
\R^{r(l/k)} = \big\{\bwy = (y_j) : \text{$j = 1, 2, \dots , r(l/k)$, and $y_j \in \R$}\big\}
\end{equation}
onto the subspace $\D_{r(l/k)}$.

\begin{lemma}\label{lemunit1}  Let $\ep_1, \ep_2, \dots , \ep_{r(l/k)}$, be a collection of multiplicatively independent elements in 
the group $\E_{l/k}(\mM)$ of relative units, and write
\begin{equation*}\label{unit80}
\ee = \big\langle \ep_1, \ep_2, \dots , \ep_{r(l/k)}\big\rangle \subseteq \E_{l/k}(\mM)
\end{equation*}
for the subgroup they generate.  Let $\bz = (z_w)$ be a vector in the subspace $\D_{r(l/k)}$.  Then there exists 
a point $\gamma$ in $\ee$ such that
\begin{equation}\label{unit85}
\sum_{w | \infty} \bigl|\log |\gamma|_w - z_w\bigr| \le \sum_{j = 1}^{r(l/k)} h(\ep_j).
\end{equation}
\end{lemma}

\begin{proof}  Let $M$ be the $(r(l) + 1) \times r(l/k)$ real matrix
\begin{equation}\label{unit90}
M = \bigl(\log |\ep_j|_w\bigr),
\end{equation}
where $w$ in $W_{\infty}(l/\Q)$ indexes rows, and $j = 1, 2, \dots , r(l/k)$, indexes columns.  Because 
$\eta_1, \eta_2, \dots , \eta_{r(l/k)}$, is a basis for the group $E_{l/k}$, there exists an $r(l/k) \times r(l/k)$ matrix 
$A = \bigl(a_{i j}\bigr)$ with integer entires such that
\begin{equation*}\label{unit95}
\log |\ep_j|_w = \sum_{i = 1}^{r(l/k)} a_{i j} \log |\eta_i|_w
\end{equation*}
for each place $w$ in $W_{\infty}(l/\Q)$ and each $j = 1, 2, \dots , r(l/k)$.  Alternatively, we have the matrix equation
\begin{equation*}\label{100}
M = L A.
\end{equation*} 
By hypothesis $\ep_1, \ep_2, \dots , \ep_{r(l/k)}$, are multiplicatively independent elements of $E_{l/k}$.  It follows that $A$ is
nonsingular, and $M$ has rank $r(l/k)$.  Using (\ref{unit53}) we conclude that
\begin{equation*}\label{unit105}
\bwy \mapsto A \bwy \mapsto L A\bwy = M \bwy = \biggl(\sum_{j = 1}^s y_j \log |\ep_j|_v \biggr)
\end{equation*}
is a linear map from the $\R$-linear space (\ref{unit58}) onto the subspace $\D_{r(l/k)}$.  In particular, there exists a unique
point $\bu = (u_j)$ in (\ref{unit58}) such that
\begin{equation}\label{unit110}
z_w = \sum_{j = 1}^{r(l/k)} u_j \log |\ep_j|_w
\end{equation}
at each place $w$ in $W_{\infty}(l/\Q)$.  Let $\bm = (m_j)$ in $\Z^{r(l/k)}$ satisfy
\begin{equation}\label{unit115}
|m_j - u_j| \le \hh,\quad\text{for each $j = 1, 2, \dots , r(l/k)$}.
\end{equation}
Then write
\begin{equation}\label{unit120}
\gamma = \ep_1^{m_1} \ep_2^{m_2} \cdots \ep_s^{m_s},\quad\text{where $s = r(l/k)$},
\end{equation}
so that $\gamma$ belongs to the subgroup $\ee$.  Using (\ref{unit35}), (\ref{unit110}), (\ref{unit115}), and
(\ref{unit120}), we find that
\begin{equation*}\label{unit125}
\begin{split}
\sum_{w | \infty} \bigl|\log |\gamma|_w - z_w\bigr| 
	&= \sum_{w | \infty} \biggl|\sum_{i = 1}^{r(l/k)} m_i \log |\ep_i|_w - \sum_{j = 1}^{r(l/k)} u_j \log |\ep_j|_w\biggr|\\
	&\le \sum_{w | \infty} \sum_{j = 1}^{r(l/k)} |m_j - u_j| \bigl|\log |\ep_i|_w\bigr|\\
	&\le \hh \sum_{w | \infty} \sum_{j = 1}^{r(l/k)} \bigl|\log |\ep_j|_w\bigr|\\
	&= \sum_{j = 1}^{r(l/k)} h(\ep_j).
\end{split}
\end{equation*}
This proves the lemma.
\end{proof}

\begin{lemma}\label{lemunit2}  Let $\ep_1, \ep_2, \dots , \ep_{r(l/k)}$, be a collection of multiplicatively independent elements in 
the group $\E_{l/k}(\mM)$ of relative units, and write
\begin{equation*}\label{unit140}
\ee = \big\langle \ep_1, \ep_2, \dots , \ep_{r(l/k)}\big\rangle \subseteq \E_{l/k}(\mM)
\end{equation*}
for the subgroup they generate.  If $\mu$ belongs to $l^{\times}$, then there exists $\gamma$ in $\ee$ such that
\begin{equation}\label{unit145}
\sum_{v | \infty} \sum_{w | v} \Bigl|\log |\gamma \mu|_w - |W_v(l/k)|^{-1} \sum_{x | v} \log |\gamma \mu|_x\Bigr|
	\le \sum_{j = 1}^{r(l/k)} h(\ep_j).
\end{equation}
\end{lemma}

\begin{proof}  Let $\bz = (z_w)$ be the vector in $\R^{r(l) + 1}$ defined at each place $w$ in $W_v(l/k)$ by
\begin{equation}\label{unit150}
z_w = |W_v(l/k)|^{-1} \sum_{x|v} \log |\mu|_x - \log |\mu|_w.
\end{equation}
It follows that at each place $v$ in $W_{\infty}(k/\Q)$ we have
\begin{equation*}\label{unit155}
\begin{split}
\sum_{w | v} z_w &= \sum_{w | v} \biggl( |W_v(l/k)|^{-1} \sum_{x|v} \log |\mu|_x - \log |\mu|_w\biggr)\\
	&= \sum_{x | v} \log |\mu|_x - \sum_{w | v} \log |\mu|_w\\
	&= 0. 
\end{split}
\end{equation*}
Therefore $\bz = (z_w)$ belongs to the subspace $\D_{r(l/k)}$.  By Lemma \ref{lemunit1} there exists an 
elements $\gamma$ in $\ee$ such that
\begin{equation}\label{unit165}
\sum_{w | \infty} \bigl|\log |\gamma|_w - z_w\bigr| \le \sum_{j = 1}^{r(l/k)} h(\ep_j).
\end{equation}
If $w | v$ then using (\ref{unit43}) and (\ref{unit150}), we find that
\begin{equation}\label{unit170}
\begin{split}
\bigl|\log |\gamma|_w - z_w\bigr| &= \Bigl|\log |\gamma \mu|_w - |W_v(l/k)|^{-1} \sum_{x | v} \log |\mu|_x\Bigr|\\
						  &= \Bigl|\log |\gamma \mu|_w - |W_v(l/k)|^{-1} \sum_{x | v} \log |\gamma \mu|_x\Bigr|.
\end{split}
\end{equation}
The inequality (\ref{unit145}) follows by combining (\ref{unit165}) and (\ref{unit170}).
\end{proof}

\section{Proof of Theorem \ref{thmintro1} and Theorem \ref{thmintro2}}

We suppose that  the full $O_k$-module $\mM \subseteq O_l$ is defined as in (\ref{intro35}), 
and that the rank $r(l/k)$ of the group $\E_{l/k}(\mM)$ of relative units is positive.  Let
\begin{equation*}\label{unit160}
\ee = \big\langle \ep_1, \ep_2, \dots , \ep_{r(l/k)}\big\rangle \subseteq \E_{l/k}(\mM)
\end{equation*}
be the subgroup generated by the multiplicatively independent units (\ref{intro95}), and assume that 
$\beta \not= 0$ in $O_k$, and $\mu \not= 0$ in $\mM$, satisfy (\ref{intro100}).

Let $\gamma$ be a point in $\ee$ such that the inequality (\ref{unit145}) holds.  Then at each place $v$ of $k$ we have
\begin{equation}\label{unit200}
\begin{split}
[l : k] \sum_{w | v} \log |\gamma \mu|_w &= \log \bigl|\Norm_{l/k}(\gamma \mu)\bigr|_v\\
							     &= \log \bigl|\Norm_{l/k}(\mu)\bigr|_v\\
							     &= \log |\beta|_v.
\end{split}
\end{equation}
We also have
\begin{equation}\label{unit205}
2 [l : k] h(\gamma \mu) = [k : l] \sum_{w | \infty} \bigl|\log |\gamma \mu|_w\bigr| 
	+ [k : l] \sum_{w \nmid \infty} \bigl|\log |\gamma \mu|_w\bigr|.
\end{equation}
Using (\ref{unit145}) and (\ref{unit200}) we estimate the first sum on the right of (\ref{unit205}) by
\begin{equation}\label{unit210}
\begin{split}
[l : k] \sum_{w | \infty} \bigl|\log |\gamma \mu|_w\bigr|
         &\le [l : k] \sum_{v | \infty} \sum_{w | v} \Bigl|\log |\gamma \mu|_w - |W_v(l/k)|^{-1} \sum_{x | v} \log |\gamma \mu|_x\Bigr|\\
         &\qquad\qquad + [l : k] \sum_{v | \infty} \sum_{w | v} |W_v(l/k)|^{-1} \Bigl|\sum_{x | v} \log |\gamma \mu|_x\Bigr|\\
         &\le [l : k] \sum_{j = 1}^{r(l/k)} h(\ep_j) + \sum_{v | \infty} \bigl|\log |\Norm_{l/k}(\gamma \mu)|_v\bigr|\\
         &= [l : k] \sum_{j = 1}^{r(l/k)} h(\ep_j) + \sum_{v | \infty} \bigl|\log |\beta|_v\bigr|.
\end{split}
\end{equation}

As $\mu$ and $\gamma \mu$ belong to $O_l$, we get
\begin{equation*}\label{unit215}
\log |\gamma \mu|_w \le 0
\end{equation*}
at each finite place $w$ of $l$.  Hence the second sum on the right of (\ref{unit205}) is
\begin{equation}\label{unit220}
\begin{split}
[l : k] \sum_{w \nmid \infty} \bigl|\log |\gamma \mu|_w\bigr| &= - [l : k] \sum_{v \nmid \infty} \sum_{w | v} \log |\gamma \mu|_w\\
	&= - \sum_{v \nmid \infty} \log |\Norm_{l/k}(\gamma \mu)|_v\\
	&= \sum_{v \nmid \infty} \bigl|\log |\beta|_v\bigr|
\end{split}
\end{equation}
By combining (\ref{unit205}), (\ref{unit210}), and (\ref{unit220}), we find that
\begin{equation}\label{unit225}
\begin{split}
2 [l : k] h(\gamma \mu) &\le [l : k] \sum_{j = 1}^{r(l/k)} h(\ep_j) + \sum_v \bigl|\log |\beta|_v\bigr|\\
	                             &= [l : k] \sum_{j = 1}^{r(l/k)} h(\ep_j) + 2 h(\beta).
\end{split}
\end{equation}
The inequality (\ref{unit225}) is also (\ref{intro110}) in the statement of Theorem \ref{thmintro1}.

Next we prove Theorem \ref{thmintro2}, where we assume that the rank of $\E_{l/k}(\mM)$ is zero.  That is, we assume
that the rank $r(k)$ is equal to the rank $r(l)$.  In general we have $r(k) \le r(l)$, and we recall 
(see \cite[Proposition 3.20]{narkiewicz2010}) that $r(k) = r(l)$ if and only if $l$ is a $\CM$-field, and $k$ is the maximal totally 
real subfield of $l$.  Assume that $\beta \not= 0$ in $O_k$, and $\mu \not= 0$ in $\mM$, satisfy (\ref{intro130}).  As in
(\ref{unit200}) we have
\begin{equation}\label{unit250}
[l : k] \sum_{w | v} \log |\mu|_w = \log |\Norm_{l/k}(\mu)|_v = \log |\beta|_v
\end{equation}
at each place $v$ of $k$.  Because $l$ is a $\CM$-field and $k$ is the maximal totally real subfield of $l$, for each archimedean 
place $v$ of $k$ the set $W_v(l/k)$ contains exactly one place of $l$.  If for each archimedean place $v$ of $k$ we write
\begin{equation*}\label{unit255}
W_v(l/k) = \{w_v\},
\end{equation*}
then (\ref{unit250}) asserts that
\begin{equation}\label{unit260}
[l : k] \log |\mu|_{w_v} = \log |\beta|_v.
\end{equation}
In particular, at each archimedean place $v$ of $k$ we get
\begin{equation}\label{unit260}
[l : k] \log^+ |\mu|_{w_v} = \log^+ |\beta|_v.
\end{equation}
As $\beta \not= 0$ and $\mu \not= 0$ are algebraic integers, we have
\begin{equation}\label{unit265}
\log |\beta|_v \le 0,\quad\text{and}\quad \log |\mu|_w \le 0,
\end{equation}
at each nonarchimedean place $v$ of $k$, and each nonarchimedean place $w$ of $l$.  Now (\ref{unit260}) and (\ref{unit265})
imply that
\begin{equation}\label{unit270}
[l : k] h(\mu) = [l : k] \sum_{v | \infty}\log^+ |\mu|_{w_v} = \sum_{v | \infty} \log^+ |\beta|_v = h(\beta).
\end{equation}
This verifies the identity (\ref{intro135}).


\end{document}